\documentclass[11pt]{article}

\usepackage{graphicx}
\usepackage{amssymb,amsmath,bm}
\usepackage{amsthm}
\usepackage{color}
\usepackage{algorithmic}
\usepackage{algorithm}

\usepackage{amsfonts}
 \usepackage{amsthm}
\usepackage{booktabs}
\usepackage{float}

\usepackage{caption}

\usepackage[bottom]{footmisc}

\usepackage[left,running,displaymath,mathlines]{lineno}

\usepackage{hyperref}

\newtheorem{remark}{Remark}

\newtheorem{theorem}{Theorem}

\def\R{\mathbb{R}}

\newtheorem{proposition}{Proposition}[section]

\usepackage{authblk}
\date{}
\begin{document}

\title{Direct transformation from Cartesian into geodetic coordinates on a triaxial ellipsoid}

\author[1]{Gema M Diaz--Toca\thanks{gemadiaz@um.es}}
\author[1]{Leandro Marin \thanks{leandro@um.es}}
\author[2]{Ioana Necula\thanks{inecula@us.es}}
\affil[1]{Departamento de Ingenier\'ia y Tecnolog\'ia de Computadores, Universidad de Murcia, Spain}
\affil[2]{Departamento de Matem\'atica Aplicada I, Universidad de Sevilla, Spain}



\maketitle

\begin{abstract}
This paper presents two new direct symbolic-numerical algorithms for the transformation of Cartesian coordinates into geodetic coordinates considering the general case of a triaxial reference ellipsoid. The problem in both algorithms is reduced to finding a real positive root of a sixth degree polynomial. The first approach consists of algebraic manipulations of the equations describing the geometry of the problem and the second one uses Gr\"obner bases. In order to perform numerical tests and accurately compare efficiency and reliability, our algorithms together with the iterative methods presented by M. Ligas (2012) and J. Feltens (2009) have been implemented in C++. The numerical tests have been accomplished by considering 10 celestial bodies, referenced in the available literature. The obtained results clearly show that our algorithms improve the aforementioned iterative methods, in terms of both efficiency and accuracy. 
\end{abstract}

{\bf Keywords:} Coordinate transformation, Cartesian coordinates, Geodetic coordinates, Triaxial ellipsoid, Symbolic-numerical computation

\section{Introduction}\label{SecIntroduction}

Transformation between Cartesian and Geodetic coordinates is an important, basic problem frequently encountered in Astronomy, Geodesy and Geoinformatics. Both coordinates are defined with respect to a Cartesian reference system and, in the case of geodetic coordinates, an ellipsoid with the center at the origin of the Cartesian reference system is also considered. Although computing Cartesian coordinates from geodetic coordinates can be easily performed, the inverse transformation is a non-trivial, challenging problem. 

\noindent In our opinion, efficient innovative solutions of this problem, as well as another actual challenges faced in Geodesy and Geoinformatics reside in the application of algebraic computational techniques combined, if necessary, with numerical methods (see, for instance, \cite{awange18}).

In the particular case of a reference biaxial ellipsoid, numerous solutions have been proposed (see, for instance, \cite{feltens08}, \cite{fukushima99} and \cite{fukushima06} for iterative solutions, \cite{turner11} for perturbation techniques based solutions and \cite{bowring76}, \cite{laloirene09} and \cite{vermeille02} for closed form solutions). Interesting solutions have been recently developed in \cite{shu10}, \cite{soler12} and \cite{civicioglu12}.

Using as geometric model of the Earth a biaxial ellipsoid is barely justified by the computational simplicity of the approach, the existing standard reference systems (such as WGS 84) and the small difference between the axes in the equatorial plane (which rounds up to 69 m). Nevertheless, the triaxiality of the Earth has been studied in many papers during the last decades (see for instance \cite{bursa71}, \cite{bursa80}, \cite{heiskanen62} and \cite{souchay03}). Moreover, in \cite{grafarend14}, the authors explicitly state (on page 862), refering to the Earth's shape parameter:``Actually, with respect to the biaxial ellipsoid, fitting the triaxial ellipsoid is 65\% better.''

Therefore, the Earth and other celestial bodies (some of them listed in Table \ref{ejes}) can be much more appropriately (in terms of accuracy of the geometric model) approximated by triaxial ellipsoids.  Furthermore, nowadays computational tools allow us to overcome the difficulty of working with three different semiaxes.

Historically, the Earth and celestial bodies with rather small diferences between semiaxes, had initially been modelled by spheres, afterwards by biaxial ellipsoids and nowadays the triaxial ellipsoid modelling is emerging. In our opinion, it might be just a matter of time until standard reference systems have based on triaxial ellipsoid.

At our best knowledge, the general case of triaxial reference ellipsoid has been considered up to the moment only in \cite{feltens09} and \cite{ligas}, both approaches giving iterative solutions.  We present in this paper two new direct symbolic-numerical algorithms giving closed form solutions, which can be applied also to a biaxial reference ellipsoid.

Therefore, the novelty of our approaches resides in tackling the issue from the \textit{symbolic} perspective, accompanied by better efficiency and accuracy results in comparison with the iterative methods developed in \cite{feltens09} and \cite{ligas}, and in using a \textit{triaxial} reference ellipsoid. The symbolic perspective consists in generating some sixth degree polynomials, prove that they have only one positive root and afterwards compute them. In the proof of the uniqueness of the positive roots, the coefficients of these polynomials are not numerical values, but symbolic, generical expressions depending on the semiaxes of the reference ellipsoid and the cartesian coordinates of the considered point.

More concretely, in the algorithm called {\tt Cartesian into Geodetic I}, described in Section \ref{SecBell}, our closed form solution consists of finding the real positive root of a sixth degree polynomial in a variable $t$. This variable $t$ serves to describe the cartesian coordinates of the given point. On the other hand, the algorithm called {\tt Cartesian into Geodetic II}, described in Section \ref{SecGB}, also consists of finding the real positive root of a sixth degree polynomial but in the variable $z$, which represents the third coordinate of the three-dimensional coordinate system.

The structure of the paper is as follows: Section \ref{prel} introduces some preliminaries and definitions. Sections \ref{SecBell} and \ref{SecGB} introduce the results that lead us to the algorithms materialized at the end of each section. Each algorithm is based on the numeric computation of the unique real positive root of a sixth degree polynomial. Both polynomials are symbolically generated: in the first approach by algebraic manipulations of the equations describing the geometry of the problem and in the second approach by computing a Gr\"obner basis. The uniqueness of the real positive roots is proven symbolically, by applying Descartes' rule of signs and studying the relative positions of several ellipsoids. The algorithm presented in Section \ref{SecBell} computes firstly the parametric coordinate (a parameter which serves to describe the cartesian coordinates) of the given point and secondly the Cartesian coordinates of the corresponding footpoint (the intersection point of the ellipsoidal normal vector passing through the given point and the ellipsoid). The algorithm presented in Section \ref{SecGB} computes firstly the $z$ coordinate of the corresponding footpoint and secondly its $x$ and $y$ coordinates. 
The numerical tests performed with the celestial bodies listed in Table \ref{ejes}, together with the obtained results, are presented in Section \ref{SecResults}. In Section \ref{SecConclusions} we present the main conclusions and further work.

\section{Preliminaries}\label{prel}

Given a point $P_E$ on a triaxial ellipsoid, its Cartesian coordinates $(X_E,Y_E,Z_E)$ satisfy the ellipsoid equation $$
f(X,Y,Z)=\frac{X^2}{a_x^2}+ \frac{Y^2}{a_y^2}+ \frac{Z^2}{a_z^2}-1=0$$
and its geodetic and Cartesian coordinates are related as follows (see \cite{muller}): $$X_E = \nu\, \cos\varphi \cos\lambda ,\quad Y_E = \nu\,(1-e_e^2)\cos \varphi \sin\lambda , \quad Z_E = \nu\, (1-e^2_x)\sin\varphi,$$ where $\nu$ is equal to the radius of the prime vertical, $\nu = \dfrac{a_x}{\sqrt{1-e_x^2\sin^2\varphi-e_e^2\cos^2\varphi\sin^2\lambda}},$
and the first eccentricities squared are $$
e_x^2=\frac{a_x^2-a_z^2}{a_x^2}, \,e_y^2=\frac{a_y^2-a_z^2}{a_y^2},\,e_e^2=\frac{a_x^2-a_y^2}{a_x^2}.
$$
Obviously, if latitude $\varphi$ and longitude $\lambda$ are given, one obtains $(X_E,Y_E,Z_E)$ by substitutions. Viceversa, if the coordinates $(X_E,Y_E,Z_E)$ are given, then
\begin{eqnarray}\label{CtoG}
\lambda & = &  
\left \{\begin{matrix} 
        \arctan\left(\dfrac{1}{(1-e_e^2)}\dfrac{Y_E}{X_E}\right),\hfill &\text{ if } X_E > 0 \hfill \nonumber \\
        \arctan\left(\dfrac{1}{(1-e_e^2)}\dfrac{Y_E}{X_E}\right) +\pi, & \text{ if }X_E<0 \hfill \nonumber \\
        \mathrm{sign}(Y_E)\,\displaystyle{\frac{\pi}{2}}, \hfill & \text{ if }X_E=0 \text{ and }Y_E\neq 0 \hfill \nonumber \\
       \mathrm{undefined}, \hfill &  \text{ if }X_E=Y_E= 0 \hfill \nonumber\\
\end{matrix} \right . \\
\varphi & = &  
\left \{\begin{matrix} 
        \arctan\left(\dfrac{(1-e_e^2)}{(1-e_x^2)}\dfrac{Z_E}{\sqrt{(1-e_e^2)^2X_E^2+Y_E^2}}\right), \hfill & \text{ if } X_E\neq 0 \text{ or } Y_E\neq 0 \hfill \nonumber
\\
        \mathrm{sign}(Z_E)\,\frac{\pi}{2}, \hfill & \text{ if }X_E=Y_E= 0 \hfill \nonumber \\
\end{matrix} \right . \\
\end{eqnarray}

\noindent However, suppose now that we have the cartesian coordinates of a point $P_G$ and we want to compute its geodetic coordinates. In this case, there exists an ellipsoidal height $h$ (see Figure \ref{hi}) such that 
\begin{equation}
\label{PG}
X_G = (\nu+h)\, \cos\varphi \cos\lambda , \quad Y_G = (\nu\,(1-e_e^2)+h)\cos \varphi \sin\lambda , \quad Z_G=(\nu\, (1-e^2_x)+h)\sin\varphi, 
\end{equation}
and the point $P_G$ will have the same latitude and longitude as the intersection point of the ellipsoidal normal vector passing through $P_G$ and the ellipsoid. This point will be named the footpoint of $P_G$. Hence, obtaining the geodetic coordinate $(\varphi,\lambda,h)$ from the Cartesian ones involves first to compute $(X_E,Y_E,Z_E)$, the footpoint of $P_G$, and secondly to apply formulas (\ref{CtoG}). 

The problem of computing the footpoint can be considered as the study of the distance from a point to an ellipsoid, a classical issue in Geometry, and it is tackled for example in \cite{bell20},\cite{hart94} and \cite{Eberly2006DistanceFA} from a less algebraic point of view than ours. Concretely, in \cite{bell20} the formula (\ref{laecu}) appears (on pages 112-113), but with practically no considerations about its resolution. \cite{hart94} is interesting as a basic, seminal approach but it seems that the conclusions are drawn without much mathematical rigor. \cite{Eberly2006DistanceFA} is a much more interesting work, Eberly considered a function defined by formula (\ref{laecu}) in our paper and analitically proved, by a Bolzano type theorem, that it had only one root in certain interval. 

\begin{figure}[H]
\begin{center}
\includegraphics[height=6cm]{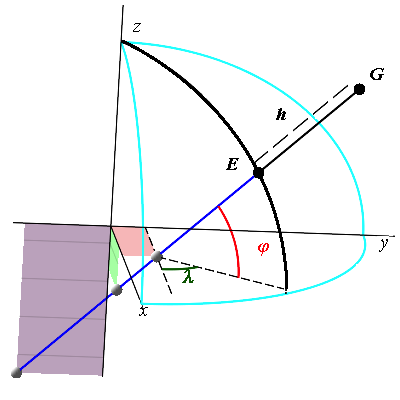}
\caption{Geometry of the problem}
\label{hi}
\end{center}
\end{figure}

\section{Computing the footpoint. First approach}\label{SecBell}

In our computations, we will apply Descartes' rule of signs, which determines the number of positive real roots of a univariate polynomial, and is based on the number of sign changes of its real coefficients.

\begin{theorem}\label{mig}[\cite{mignotte} Descartes' rule] Let $f(X)= a_n X^n+a_{n-1}X^{n-1}+\cdots+a_0$ be a
polynomial in $\R[x]$, where $a_n$ and $a_0$ are nonzero. Let $v$ be the
number of changes of signs in the sequence $[a_{n},\ldots ,a_0]$ of its coefficients
and let $r$ be the number of its real positive roots, counted with their orders
of multiplicity. Then there exists some nonnegative integer $m$ such
that $r = v -2m.$
\end{theorem}

\noindent We will apply Descartes' rule several times across the paper, for polynomials whose number of sign changes in its lists of coefficients is equal to 0 or 1, therefore they have no or one positive real root, respectively. Analyzing the sign of the coefficients of these polynomials will be reduced to studying the relative positions of several ellipsoids. These ellipsoids have the same center and each ellipsoid will turn out to be placed inside or outside the others, having no intersection points. 

The unique positive real roots of these polynomials will be used to determine the footpoint of a given point (see Equations (\ref{XEfromXG}) and (\ref{XEfromGb})). 

We assume throughout the paper, for simplicity, that our point $P_G\neq (0,0,0)$ is situated in the first octant and also that $a_x>a_y>a_z$. We define $P=(a_x-a_z)(a_x+a_z)>0 ,  Q=(a_y-a_z)(a_y+a_z)>0  \text{ and } R=(a_x-a_y)(a_x+a_y)>0$. 

\noindent Following \cite{bell20}, \cite{feltens09} and \cite{ligas}, the gradient of $f(X,Y,Z)$ evaluated in the footpoint $P_E$ provides a normal vector to the ellipsoid, $\vec{n}=2\left( \dfrac{X_E}{a_x^2},\dfrac{Y_E}{a_y^2},\dfrac{Z_E}{a_z^2}\right),$ and a vector connecting point $P_G$ and $P_E$ is $$\vec{h}=(X_G-X_E,Y_G-Y_E,Z_G-Z_E)=h( \cos\varphi \cos\lambda,\cos \varphi \sin\lambda,\sin\varphi) $$ with $P_G= \vec{h} +  P_E$. Both vectors $ {\vec{h}}$ and $ {\vec{n}}$ must be proportional and so, in the general case $|h|>0$, there is a real value $t$ with $$\label{EcT}
t= \dfrac{X_G-X_E}{X_E/a_x^2}=\dfrac{Y_G-Y_E}{Y_E/a_y^2}=\dfrac{Z_G-Z_E}{Z_E/a_z^2},
$$ 
and thus 
\begin{equation}\label{XEfromXG}
X_E = \frac{a_x^2\, X_G }{t + a_x^2},\,Y_E = \frac{a_y^2\, Y_G }{t + a_y^2},\,Z_E = \frac{a_z^2\, Z_G }{t + a_z^2} 
\end{equation}
Since $\dfrac{X_E^2}{a_x^2}+ \dfrac{Y_E^2}{a_y^2}+ \dfrac{Z_E^2}{a_z^2}=1,$ 
we have 
\begin{equation}
\label{laecu}
 \dfrac{(a_x \, X_G)^2 }{(t + a_x^2)^2 }+\dfrac{(a_y \, Y_G)^2 }{(t + a_y^2)^2}+\dfrac{(a_z \, Z_G)^2 }{(t + a_z^2)^2 }-1=0.
\end{equation}
The numerator of Equation (\ref{laecu}) is the polynomial $A(t)= t^6+A_5t^5+A_4t^4+A_3t^3+A_2t^2+A_1t+A_0$, where
\begin{eqnarray*}
A_5 &=&2\,(a_x^2 + a_y^2+a_z^2)>0,\\  
A_4&=&-a_x^2X_G^2-a_y^2Y_G^2-a_z^2Z_G^2+(a_x^2+a_y^2+a_z^2)^2+2(a_x^2a_y^2+a_x^2a_z^2	+a_y^2a_z^2),\\ 
A_3&=&-2\left(\, a_x^2(a_y^2+a_z^2)X_G^2+a_y^2(a_x^2+a_z^2)Y_G^2+a_z^2(a_x^2+a_y^2)Z_G^2- \right.\\
&&  \left.- (a_x^2+a_y^2+a_z^2)(a_y^2a_z^2+a_x^2a_y^2+a_x^2a_z^2) - a_x^2a_y^2a_z^2 \,\right),
\\  
A_2&=&-a_x^2(a_y^4+4a_y^2a_z^2+a_z^4)X_G^2-a_y^2(a_x^4+4a_x^2a_z^2+a_z^4)Y_G^2-a_z^2(a_x^4+4a_x^2a_y^2+a_y^4)Z_G^2+\\
&&+ (a_x^2a_y^2+a_x^2a_z^2+a_z^2a_y^2)^2 + 2a_x^2a_y^2a_z^2(a_x^2+a_y^2+a_z^2),
\\ 
A_1&=&-2a_x^2a_y^2a_z^2\left(\,(a_y^2+a_z^2)X_G^2+(a_x^2+a_z^2)Y_G^2+(a_x^2+a_y^2)Z_G^2-a_x^2a_y^2-a_x^2a_z^2-a_y^2a_z^2 \,\right),\\  
A_0&=& -a_x^2a_y^2a_z^2\left(\,a_x^2a_y^2Z_G^2+a_x^2a_z^2Y_G^2+a_y^2a_z^2X_G^2-a_x^2a_y^2a_z^2\,\right).\\  
\end{eqnarray*}
The variable $t$ can be considered as a parametric coordinate of $P_G$ and is positive if the point is situated outside the reference ellipsoid, negative if it is situated inside or 0 if it is situated on the reference ellipsoid. Obviously, the ellipsoidal heigh $h$ is equal to 0 iff $A_0=0$.

\begin{remark}\label{casoBiaxial}
In the particular case of a biaxial reference ellipsoid, when $a_x=a_y$, the Equation (\ref{laecu}) becomes
\begin{equation}
\dfrac{(a_x \, X_G)^2 + (a_x \, Y_G)^2}{(t + a_x^2)^2 }+\dfrac{(a_z \, Z_G)^2 }{(t + a_z^2)^2 }-1=0
\end{equation}
and leads to the fourth degree polynomial $\alpha(t)=t^4 + \alpha_3 t^3 + \alpha_2 t^2 + \alpha_1 t + \alpha_0$ where 
\begin{eqnarray*}
\alpha_3 &=&2\left({{\it a_x}}^2+{{\it a_z}}^2\right),\\
\alpha_2 &=& -{{\it a_x}}^{2}({{\it X_G}}^{2}+{{\it Y_G}}^{2})-{{\it a_z}}^{2}{{\it Z_G}}^
{2}+({{\it a_x}}^{2}+{{\it a_z}}^{2})^2+2\,{{\it a_x}}^{2}{{\it a_z}}^{2},\\
\alpha_1 &=&-2\,{{\it a_x}}^{2}{{\it a_z}}^{2} \left( {{\it X_G}}^{2}+{{\it Y_G}}^{2}+
{{\it Z_G}}^{2}-{{\it a_x}}^{2}-{{\it a_z}}^{2} \right), \\
\alpha_0 &=& -{{\it a_x}}^{2}{{\it a_z}}^{2} \left( {{\it a_z}}^{2}{{\it X_G}}^{2}+{{\it a_z}}^{2}{{
\it Y_G}}^{2}+{{\it a_x}}^{2}{{\it Z_G}}^{2}-{{\it a_x}}^{2}{{\it a_z}}^{2} \right).\\
\end{eqnarray*}
The results obtained in this paper can be established also for the biaxial case. Nevertheless, we do not consider of any relevance this particular case: the aforementioned fourth degree polynomial has been studied in \cite{laloirene09} completely symbolically, by using Sturm--Habicht coefficients and subresultants, having led to a close form solution. 
\end{remark}



\begin{proposition}\label{propSignoBell}
The number of sign changes in $[A_5,A_4,A_3,A_2,A_1,A_0]$ is equal to 1 if the point $P_G$ is situated outside the reference ellipsoid, or 0 if the point $P_G$ is situated inside or on the reference ellipsoid.
\end{proposition}
\begin{proof}
The sign of $A_0$ depends on the sign of the factor 
$$a_x^2a_y^2Z_G^2+a_x^2a_z^2Y_G^2+a_y^2a_z^2X_G^2-a_x^2a_y^2a_z^2,
$$ 
which is the numerator of $f(X_G,Y_G,Z_G)-1$. 
The sign of $A_1$ depends on the sign of the factor
$$
 (a_y^2+a_z^2)X_G^2+(a_x^2+a_z^2)Y_G^2+(a_x^2+a_y^2)Z_G^2-a_x^2a_y^2-a_x^2a_z^2-a_y^2a_z^2,  
 $$
which defines the ellipsoid of equation
\begin{eqnarray*}\label{c1_elip}
e_{1}: & X^2\dfrac{   a_y^2+a_z^2   }{a_x^2  a_y^2 +  a_x^2 a_z^2+  a_y^2 a_z^2} +Y^2\dfrac{  a_x^2+a_z^2  }{a_x^2  a_y^2 +  a_x^2 a_z^2+  a_y^2 a_z^2}  +Z^2\dfrac{ a_x^2+a_y^2}{a_x^2  a_y^2 +  a_x^2 a_z^2+  a_y^2 a_z^2} = 1. \hfill
\end{eqnarray*}
Since 
$$\dfrac{a_x^2  a_y^2 +  a_x^2 a_z^2+  a_y^2 a_z^2}{   a_y^2+a_z^2   }>a_x^2, \ \ \dfrac{a_x^2  a_y^2 +  a_x^2 a_z^2+  a_y^2 a_z^2}{ a_x^2+a_z^2   }>a_y^2, \ \ 
\dfrac{a_x^2  a_y^2 +  a_x^2 a_z^2+  a_y^2 a_z^2}{   a_x^2+a_y^2   }>a_z^2,$$
the original, reference ellipsoid $e_{original}$ is situated inside the ellipsoid $e_{1}$.

\noindent The coefficient $A_2$ defines the ellipsoid of equation
\begin{eqnarray*}\label{c2_elip}
e_{2}:& X^2 \dfrac{a_x^2(a_y^4+4a_y^2a_z^2+a_z^4)}{(a_x^2a_y^2+a_x^2a_z^2+a_z^2a_y^2)^2 + 2a_x^2a_y^2a_z^2(a_x^2+a_y^2+a_z^2)} +\hfill &\\
& +Y^2\dfrac{a_y^2(a_x^4+4a_x^2a_z^2+a_z^4)}{(a_x^2a_y^2+a_x^2a_z^2+a_z^2a_y^2)^2 + 2a_x^2a_y^2a_z^2(a_x^2+a_y^2+a_z^2)}+ \hfill & \\ 
& +Z^2\dfrac{a_z^2(a_x^4+4a_x^2a_y^2+a_y^4)}{(a_x^2a_y^2+a_x^2a_z^2+a_z^2a_y^2)^2 + 2a_x^2a_y^2a_z^2(a_x^2+a_y^2+a_z^2)}\hfill &=1.\hfill
\end{eqnarray*} 
The semiaxes of the ellipsoid $e_{2}$ are bigger than the corresponding semiaxes of the ellipsoid $e_{1}$, and in consequence $$ e_{original} \subset e_{1}  \subset e_{2}.$$

\noindent The sign of the coefficient $A_3$ depends on a negative factor and on the factor $$ a_x^2(a_y^2+a_z^2)X_G^2+a_y^2(a_x^2+a_z^2)Y_G^2+a_z^2(a_x^2+a_y^2)Z_G^2- (a_x^2+a_y^2+a_z^2)(a_y^2a_z^2+a_x^2a_y^2+a_x^2a_z^2) - a_x^2a_y^2a_z^2.$$ This factor defines the ellipsoid of equation
\begin{eqnarray*}\label{c3_elip}
e_{3}: & X^2\dfrac{a_x^2(a_y^2+a_z^2)}{(a_x^2+a_y^2+a_z^2)(a_y^2a_z^2+a_x^2a_y^2+a_x^2a_z^2) + a_x^2a_y^2a_z^2}+\hfill &  \\ & +Y^2\dfrac{a_y^2(a_x^2+a_z^2)}{(a_x^2+a_y^2+a_z^2)(a_y^2a_z^2+a_x^2a_y^2+a_x^2a_z^2) + a_x^2a_y^2a_z^2}+\hfill & \\ & +Z^2\dfrac{a_z^2(a_x^2+a_y^2)}{(a_x^2+a_y^2+a_z^2)(a_y^2a_z^2+a_x^2a_y^2+a_x^2a_z^2) + a_x^2a_y^2a_z^2}\hfill &=1.\hfill 
\end{eqnarray*}
The semiaxes of the ellipsoid $e_{3}$ are also bigger than the corresponding semiaxes of the ellipsoid $e_{2}$, and in consequence $$ e_{original} \subset e_{1}  \subset e_{2} \subset e_{3}.$$

\noindent Finally, the coefficient $A_4$ defines the ellipsoid of equation
\begin{eqnarray*}\label{c4_elip}
e_{4}: & X^2\dfrac{a_x^2}{(a_x^2+a_y^2+a_z^2)^2+2(a_x^2a_y^2+a_x^2a_z^2	+a_y^2a_z^2)}+ \hfill & \\ & + Y^2\dfrac{a_y^2}{(a_x^2+a_y^2+a_z^2)^2+2(a_x^2a_y^2+a_x^2a_z^+a_y^2a_z^2)}+ \hfill & \\ & +Z^2\dfrac{a_z^2}{(a_x^2+a_y^2+a_z^2)^2+2(a_x^2a_y^2+a_x^2a_z^2+a_y^2a_z^2)} \hfill &= 1 \hfill  
\end{eqnarray*}
The semiaxes of the ellipsoid $e_{4}$ are also bigger than the corresponding semiaxes of the ellipsoid $e_{3}$, and in consequence $$ e_{original} \subset e_{1}  \subset e_{2} \subset e_{3}\subset e_{4}.$$

\noindent Therefore, the signs of the list $[A_5,A_4,A_3,A_2,A_1,A_0]$ must be one of the following (being the number of sign changes equal to 1 for an outside point $P_G$ and 0 otherwise):
\begin{itemize}
\item $[+,+,+,+,+,+]$ if $P_G$ is inside the reference ellipsoid,
\item $[+,+,+,+,+,0]$ if  $P_G$ is on the reference ellipsoid,
\item $[+,+,+,+,+,-]$ if  $P_G$ is outside the reference ellipsoid and inside  $e_1$,
\item $[+,+,+,+,0,-]$ if $P_G$ is on $e_1$, 
\item $[+,+,+,+ ,-,-]$ if $P_G$ is  outside $e_1$ and inside $e_2$,
\item $[+,+,+,0,-,-]$ if $P_G$ is on $e_2$,
\item $[+,+,+,- ,-,-]$  if $P_G$ is outside $e_2$ and inside $e_3$,
\item $[+,+,0,-,-,-]$ if $P_G$ is on $e_3$,
\item $[+,+,-,- ,-,-]$  if $P_G$ is outside $e_3$ and inside $e_4$,
\item $[+,0,-,- ,-,-]$ if $P_G$ is on $e_4$,
\item $[+,-,-,- ,-,-]$ if $P_G$ is outside $e_4$.
\end{itemize}
\end{proof}

Consequently if $P_G$ is outside the reference ellipsoid, then the polynomial $A(t)$ has a unique real positive root. If $P_G$ is inside the reference ellipsoid, then the polynomial $A(t)$ has no positive real roots. If $P_G$ is on the reference ellipsoid, then it has no positive real roots and furthermore $A(0)=0$. 

%

\subsection{$P_G$ situated inside the ellipsoid}
We will analyze in the following the case of $P_G$ being situated inside the ellipsoid. Suppose first that $Z_G>0$. Then $Z_E>0$ and because of (\ref{XEfromXG}), we should have $t>-a_z^2$. Therefore, there exists $k>0$ with $t=-a_z^2+k.$ That leads us to consider the polynomial $\bar{A}(k)=A(-a_z^2+k)$, whose number of positive real roots is equal to the number of real (negative, since $A(t)$ has no positive real roots in this case) roots of $A(t)$ satisfying $t>-a_z^2$. 

\noindent By applying Descartes' rule, we will see that $\bar{A}(k)$ has only one positive root. We obtain that $\bar{A}(k)= k^6+\bar{A}_5 k^5+\bar{A}_4 k^4+\bar{A}_3 k^3+\bar{A}_2 k^2+\bar{A}_1 k+\bar{A}_0$, where 
\begin{eqnarray*}
\bar{A}_5 &=& 2(P+Q) >0,\\  
\bar{A}_4&=& - a_x^2X_G^2 - a_y^2Y_G^2 - a_z^2Z_G^2 + P^2+Q^2+4PQ,\\  
\bar{A}_3&=& 2\left( -a_x^2 Q X_G^2 - a_y^2 P Y_G^2 - a_z^2 (P+Q) Z_G^2 + PQ(P+Q)  \right),
\\  
\bar{A}_2&=& -a_x^2Q^2 X_G^2 - a_y^2 P^2 Y_G^2 -a_z^2(P^2+Q^2+4PQ) Z_G^2 + P^2Q^2,
\\  
\bar{A}_1&=& -2a_z^2PQ(P+Q) Z_G^2\leq 0\,,
\\   
\bar{A}_0&=& -a_z^2P^2 Q^2 Z_G^2\leq 0\,.  
\end{eqnarray*}

\begin{proposition}\label{propSignoBell1}
 If $Z_G>0$, the number of sign changes in the list $[\bar{A}_5,\bar{A}_4,\bar{A}_3,\bar{A}_2,\bar{A}_1,\bar{A}_0]$ is equal to 1.
\end{proposition}

\begin{proof}
The coefficient $\bar{A}_2$ defines the ellipsoid $\bar{e}_2$,
$$\label{c2_elip1}
\bar{e}_2: \, X^2 \,\frac{a_x^2}{P^2} + Y^2\, \frac{a_y^2}{Q^2} + Z^2\, \frac{a_z^2(P^2+Q^2+4PQ)}{P^2Q^2}=1.
$$
The coefficient $\bar{A}_3$ defines the ellipsoid of equation
$$\label{c3_elip1}
\bar{e}_{3}: X^2 \frac{a_x^2}{P(P+Q)} + Y^2 \frac{a_y^2}{Q(P+Q)} + Z^2 \frac{a_z^2}{PQ}=1.
$$
The coefficient $\bar{A}_4$ defines the ellipsoid of equation 
$$\label{c4_elip1}
\bar{e}_4: X^2 \frac{a_x^2}{P^2+Q^2+4PQ} + Y^2 \frac{a_y^2}{P^2+Q^2+4PQ} + Z^2 \frac{a_z^2}{P^2+Q^2+4PQ}=1.$$
Since $$P^2<P(P+Q)<P^2+Q^2+4\,PQ , \quad Q^2<Q(P+Q)<P^2+Q^2+4\,PQ, \quad  \frac{P^2Q^2}{P^2 +Q^2  +4\,PQ }<PQ<P^2+Q^2+4\,PQ,$$ we have $\bar{e}_{2} \subset \bar{e}_{3} \subset \bar{e}_{4}$. Therefore, the signs of the list $[\bar{A}_5, \bar{A}_4, \bar{A}_3, \bar{A}_2, \bar{A}_1, \bar{A}_0]$ must be one of the following: 
\begin{itemize}
\item $[+,+,+,+,-,-]$ if the point $P_G$ is inside $\bar{e}_{2}$,
\item $[+,+,+,0,-,-]$ if the point $P_G$ is on $\bar{e}_{2}$,
\item $[+,+,+,-,-,-]$ if the point $P_G$ is outside $\bar{e}_{2}$ and inside $\bar{e}_{3}$,
\item $[+,+,0,-,-,-]$ if the point $P_G$ is on $\bar{e}_{3}$,
\item $[+,+,-,-,-,-]$ if the point $P_G$ is outside $\bar{e}_{3}$ and inside $\bar{e}_{4}$,
\item $[+,0,-,-,-,-]$ if the point $P_G$ is on $\bar{e}_{4}$,
\item $[+,-,-,-,-,-]$ if the point $P_G$ is outside $\bar{e}_{4}$.
\end{itemize}

\end{proof}
Consequently if $P_G$ is situated inside the reference ellipsoid with $Z_G>0$ then the polynomial $A(t)$ has a unique real root satisfying $-a_z^2<t<0$.

%
%
 
Suppose now that $Z_G=0$. Then, $\varphi=0$ and the footpoint $P_E$ is on the ellipse 
\begin{equation}\label{elipse}
\dfrac{X^2}{a_x^2}+ \dfrac{Y^2}{a_y^2}=1.
\end{equation}

\noindent Observe that if $Y_G=0$, then $\lambda=0$ and if $X_G=0$ then $\lambda=\dfrac{\pi}{2}$. Suppose that $X_G>0$ and $Y_G>0$. Thus, following the same reasoning as before, we will have $$\frac{(a_x \, X_G)^2 }{(t + a_x^2)^2 }+\frac{(a_y \, Y_G)^2 }{(t + a_y^2)^2}-1=0,
$$
with the numerator equal to $\Delta(t)={t}^{4}+\Delta_3t^3+\Delta_2t^2+\Delta_1t +\Delta_0$, where 
 \begin{eqnarray*}
  \Delta_3&=& 2\left(  a_x^{2}+  a_y^{2} \right)>0 ,\\
 \Delta_2&=& \left( a_x^{4}+4\, a_x^2 a_y^{2}+ a_y^{4} -a_x^{2}X_G^{2}- a_y^{2} Y_G^{2} \right) , \\
 \Delta_1&=&2\,a_x^{2} a_y^{2} \left( a_x^{2}+a_y^{2}- X_G^{2}- Y_G^{2} \right),\\
 \Delta_0&=& a_x^{2} a_y^{2} \left(  a_x^{2} a_y^{2}- a_x^{2} Y_G^{2}-
a_y^{2} X_G^{2} \right) .
\end{eqnarray*}

\noindent In this case, $\Delta_0$ is zero iff the point $P_G$ is situated on the ellipse $(\ref{elipse})$, and the number of sign changes in the list $[\Delta_3,\Delta_2,\Delta_1,\Delta_0]$ is zero for a point $P_G$ inside or on the ellipse $(\ref{elipse})$. However, by the same reasoning as before, $t$ must be bigger than $-a_y^2$ and if we substitute $k-a_y^2$ for $t$ in $\Delta(t)$, we obtain 
$$\bar{\Delta}(k)={k}^{4}+\bar{\Delta}_3k^3+\bar{\Delta}_2k^2+\bar{\Delta}_1k +\bar{\Delta}_0,$$
with 
\begin{equation}\label{deltas}
\bar{\Delta}_3= 2R>0 , \quad \bar{\Delta}_2=R^2 -a_x^2 X_G^2  -a_y^2 Y_G^2  , \quad \bar{\Delta}_1= -2 a_y^2 Y_G^2 R<0 , \quad  \bar{\Delta}_0= -a_y^2Y_G^2 R^2<0,
\end{equation} 
therefore the number of sign changes in the list $[  \bar{\Delta}_3, \bar{\Delta}_2, \bar{\Delta}_1, \bar{\Delta}_0]$ is equal to 1.

Consequently if $P_G$ is situated inside the reference ellipsoid with $Z_G=0$, $X_G>0$ and $Y_G>0$, then the polynomial $\Delta(t)$ has a unique real root satisfying $-a_y^2<t<0$.

\subsection{The algorithm}
All these results lead to the following algorithm.
%
%


\begin{algorithm}[H]
\begin{algorithmic}[1]
\REQUIRE{The semiaxes of the triaxial reference ellipsoid. \\ \quad \quad \quad The Cartesian coordinates $(X_G,Y_G,Z_G)\neq (0,0,0)$. }
\ENSURE The geodetic coordinates $(\varphi,\lambda,h)$.
\IF{$f(X_G,Y_G,Z_G)=1$} \STATE {  $(X_G,Y_G,Z_G)=(X_E,Y_E,Z_E)$, $(\varphi,\lambda)$ are computed from Equalities (\ref{CtoG}) and $h=0$;
} 
 \ELSE 
\IF{$f(X_G,Y_G,Z_G)>1$} 
 
\STATE{ evaluate coefficients $A_i$, $i=0,\ldots,5$;}\quad\COMMENT{see Proposition \ref{propSignoBell}}
\STATE{ compute $\textsc{t}$ the unique positive root of $A(t)$;} 
\STATE{  substitute $t=\textsc{t}$ in Equalities (\ref{XEfromXG}) for computing $(X_E,Y_E,Z_E)$;} 
\STATE{$h=|(X_G,Y_G,Z_G)-(X_E,Y_E,Z_E)|$    }
\ELSE
\IF{$Z_G>0$}
\STATE{ evaluate coefficients $\bar{A}_i$, $i=0,\ldots,5$;}\quad\COMMENT{see Proposition \ref{propSignoBell1}}
\STATE{ compute $\textsc{k}$ the unique positive root of $\bar{A}(k)$;}
\STATE{  substitute $t=-a_z^2+\textsc{k}$ in Equalities (\ref{XEfromXG}) for computing $(X_E,Y_E,Z_E)$;} 
\STATE{$h=-|(X_G,Y_G,Z_G)-(X_E,Y_E,Z_E)|$;    }
\STATE{compute $(\varphi,\lambda)$ from Equalities (\ref{CtoG})  }
\ELSE
\STATE{$Z_E=0$; $\varphi=0$;}
\IF {$X_G>0$, $Y_G>0$}
\STATE { evaluate coefficients $\bar{\Delta}_i$, $i=0,\ldots,3$;}\quad\COMMENT{see Equations (\ref{deltas})}
\STATE{ compute $\textsc{k}$ the unique positive root of $\bar{\Delta}(k)$;}
\STATE{  substitute $t=-a_y^2+\textsc{k}$ in Equalities (\ref{XEfromXG}) for computing $X_E$ and $Y_E$;} 
\STATE{$h=-|(X_G,Y_G)-(X_E,Y_E)|$;    }
\STATE{compute $\lambda$ from Equalities (\ref{CtoG})  }
\ENDIF
\IF {$X_G=0$} 
\STATE {$X_E=0$; $Y_E=a_y$; $\lambda=\displaystyle{\frac{\pi}{2}}$; $h=Y_G-Y_E$}
\ENDIF
\IF {$Y_G=0$}
\STATE {$X_E=a_x$; $Y_E=0$; $\lambda=0$; $h=X_G-X_E$}
\ENDIF
\ENDIF
\ENDIF
\
\ENDIF 
\end{algorithmic}
\caption*{{\bf Algorithm} {\tt Cartesian into Geodetic I}}
\end{algorithm}


\section{Computing the footpoint. Second approach}\label{SecGB}

The ideal generated by a family of polynomials is defined to be the set of linear combinations, with polynomial coefficients, of these polynomials (see \cite{cox07} pg.30 for details). If we have a system of equations with finitely many solutions, it is well known that a Gr\"obner basis (see \cite{awange18} and \cite{cox07} for details) of the ideal generated by the equations of such a system provides another equivalent system but in triangular form, which is much easier to solve. We will explore this idea in this section.

According to Section \ref{SecBell}, the cartesian coordinates of the footpoint must satisfy the system of equations in three unknowns given by: 
$$
\displaystyle{\frac{x^2}{a_x^2}}+ \displaystyle{\frac{y^2}{a_y^2}}+ \displaystyle{\frac{z^2}{a_z^2}}=1, \quad
\dfrac{X_G-x}{x/a_x^2}-\dfrac{Y_G-y}{y/a_y^2}=0,  \quad \dfrac{X_G-x}{x/a_x^2}-\dfrac{Z_G-z}{z/a_z^2}=0,  \quad \dfrac{Y_G-y}{y/a_y^2}-\dfrac{Z_G-z}{z/a_z^2}=0. 
$$
By assuming first that none of three variables is zero, this system is equivalent to the following one:
$$
S:  \left\{\begin{array}{ccc} \medskip
   a_y^2a_z^2x^2  +a_x^2a_z^2y^2+ a_x^2a_y^2z^2 - a_x^2a_y^2a_z^2&=&0,\\ \medskip
  a_x^2xy - a_x^2X_Gy - a_y^2xy + a_y^2Y_Gx&=&0,\\ \medskip
  a_x^2xz - a_x^2X_Gz - a_z^2xz + a_z^2Z_Gx&=&0 ,\\ \medskip
  a_z^2 y  z+a_y^2 Y_{G} z - a_z^2 Z_{G} y -  a_y^2 y z&=&0.\\ 
\end{array}\right.
$$
The system $S$ has finitely many solutions, and so, as mentioned previously, a Gr\"obner basis of the ideal generated by the equations of $S$ provides another equivalent system but in triangular form in the variables $x,y,z$. The univariate equation in $z$ in the Gr\"obner basis\footnote{The Gr\"obner basis using the lexicographical order with $y>x>z$ (see \cite{cox07} pg.56 for details), computed with {\tt Maple 2017} is available at \url{http://dx.doi.org/10.17632/xw5ws5gz8x.1}.} is given by $\label{Bz} B(z)= B_6 z^6 + B_5 z^5 + B_4 z^4 + B_3 z^3 + B_2 z^2 + B_1 z+B_0,$ where 
\begin{eqnarray*}
B_6 &=&P^2 Q^2>0,\\ 
B_5 &=&2\, a_z ^2  Z_G \,PQ \left( P+Q \right)\geq 0, \\  
B_4&=&a_z^2 \left(a_x^2 Q^2 X_G^2 +  a_y^2 P^2Y_G^2+a_z^2 \left(  P^2+Q^2  +4\,PQ \right) Z_G^2-P^2Q^2\right), \\  
B_3&=&2\, a_z^4 Z_G\, \left(  a_x^2 Q X_G^2 + a_y^2P\,Y_G^2 + a_z^2 \left( P+Q \right)  Z_G^2-  PQ \left( P+Q \right)  \right) ,\\  
B_2&=&a_z^6 Z_G^2 \left(   a_x^2 X_G^2 + a_y^2 Y_G^2 + a_z^2 Z_G^2-  P^2-Q^2  -4\,PQ \right) ,\\  
B_1&=&-2\,a_z^8\,Z_G^3\left( P+Q \right)\leq 0 ,\\  
B_0&=& -a_z^{10}Z_G^4\leq 0\,.   
\end{eqnarray*}
Therefore, the positive root of $B(z)$ will be the coordinate $Z_E$ required.
\begin{proposition}  \label{propG}
The  number of sign changes in the list $[B_6, B_5,B_4,B_3,B_2,B_1,B_0]$ is equal to 1 if $Z_G>0$. 
\end{proposition}

\begin{proof}
The signs of $B_2$, $B_3$ and $B_4$ are determined by the ellipsoids $\bar{e}_4$, $\bar{e}_3$ and $\bar{e}_2$, respectively, introduced in the proof of  Proposition \ref{propSignoBell1}. Since $\bar{e}_2 \subset \bar{e}_3 \subset \bar{e}_4$, if $Z_G>0$ the signs of the list $[B_6,B_5,B_4,B_3,B_2,B_1,B_0]$ must be one of the following: \begin{itemize}
\item $[+,+,-,- ,- ,-,-]$ if  $P_G$ is inside $\bar{e}_2$,
\item $[+,+,0,- ,- ,-,-]$ if  $P_G$ is on $\bar{e}_2$,
\item $[+,+,+,- ,- ,-,-]$ if  $P_G$ is outside $\bar{e}_2$ and inside $\bar{e}_3$,
\item $[+,+,+,0,-,-,-]$ if  $P_G$ is on $\bar{e}_3$,
\item $[+,+,+, +,- ,-,-]$ if  $P_G$ is outside $\bar{e}_3$ and inside $\bar{e}_4$,
\item $[+,+,+,+ ,0 ,-,-]$ if  $P_G$ is on $\bar{e}_4$,
\item $[+,+,+,+ ,+ ,-,-]$  if  $P_G$ is outside $\bar{e}_4$.
\end{itemize}
\end{proof}

\noindent Consequently, if $Z_G>0$, $B(z)$ has only one real positive root, which is equal to $Z_E$. Moreover, the polynomials 
$$
B_2(x,z)= \left( Pz + a_z^2Z_G \right) x -  a_x^2 X_G z , \quad B_3(y,z)=\left( Qz + a_z^2 Z_G  \right)  y - a_y^2 Y_G z,
$$
part of the Gr\"obner basis, provide the coordinates $X_E$ and $Y_E$:
\begin{equation}\label{XEfromGb} 
X_E = \frac{a_x^2 X_G Z_E}{\left( PZ_E + a_z^2 Z_G \right)} , \quad Y_E = \frac{a_y^2 Y_G Z_E }{\left( Q Z_E + a_z^2 Z_G \right)}.
\end{equation} 

\noindent On the other hand, if $Z_G = 0$ then $Z_E=0$ and we obtain a new system
$$ 
a_x^2 y^2 + a_y^2 x^2 - a_x^2 a_y^2=0, \quad (a_x^2-a_y^2) x y - a_x^2 X_G y + a_y^2 Y_G x=0,
$$
whose Gr\"obner basis\footnote{Available at \url{http://dx.doi.org/10.17632/xw5ws5gz8x.1}} contains the polynomials 
\begin{equation}\label{g1y}
G_1(y)=R^2 y^4 + 2 a_y ^2 R Y_G y^3 - a_y^2\left( R^2 - a_x^2X_G^2-a_y^2Y_G^2 \right) y^2 - 2a_y^4 R Y_G y - a_y^6 Y_G^2,
\end{equation} 
$$
G_2(x,y)=\left( R y+a_y^2 Y_G \right) x - a_x^2 X_G y \label{g2xy},
$$
which provide the coordinates $Y_E$ and $X_E$. As the coefficients in $y^4$ and $y^3$ of $G_1(y)$ are positive and the coefficient in $y$ and the independent one are negative, the number of changes of signs in the list of coefficients of $G_1(y)$ is equal to 1. Consequently,  $G_1(y)$ has a unique real positive root. 

\noindent Finally, if both $Z_G=0$ and $Y_G=0$ (unusual in practice) then $\varphi=\lambda=0$.


\begin{algorithm}[H]
\begin{algorithmic}[1]
\REQUIRE{The semiaxes of the triaxial reference ellipsoid. \\ \quad \quad \quad The Cartesian coordinates $(X_G,Y_G,Z_G)\neq (0,0,0)$.
}
\ENSURE  The geodetic coordinates $(\varphi,\lambda,h)$.

\IF{$Z_G\neq 0$} 
\STATE {evaluate the coefficients $B_i$, $i=0,\ldots,6$;}\quad \COMMENT{see Proposition \ref{propG}}
\STATE {compute $Z_E$ the unique positive root of $B(z)$;}
\STATE {compute $X_E$ and $Y_E$ from  Equalities (\ref{XEfromGb}); }
\STATE{compute $(\varphi,\lambda)$ from Equalities (\ref{CtoG})  }
\ELSE 
\STATE{ $Z_E=0$; $\varphi=0$;}
 \IF{$Y_G\neq 0$, } 
\STATE{evaluate the coefficients of the polynomial $G_1(y)$;} \quad \COMMENT{see Equations (\ref{g1y})}
\STATE{compute $Y_E$ the unique positive root of $G_1(y)$;}
\STATE{compute $X_E$ the unique real root of $G_2(x,Y_E)$;}
\STATE{compute $\lambda$ from Equalities (\ref{CtoG}) }
\ELSE
\STATE{ $Y_E=0$; $X_E=a_x$; $\lambda=0$ }
\ENDIF 
\ENDIF 
\IF{$f(X_G,Y_G,Z_G)\geq 1$}  
\STATE{$h=|(X_G,Y_G,Z_G)-(X_E,Y_E,Z_E)|$    }
\ELSE
\STATE{$h=-|(X_G,Y_G,Z_G)-(X_E,Y_E,Z_E)|$    }
\ENDIF
\end{algorithmic}
\caption*{{\bf Algorithm} {\tt Cartesian into Geodetic II}}
\end{algorithm}

\section{Numerical tests}\label{SecResults}

Our algorithms have been initially implemented in the Scientific Computing System {\tt Maple 2017}. We have implemented also the methods presented in \cite{feltens09} and \cite{ligas}, in order to accurately compare the results (maximum errors and running times). This initial study showed that the best running times and the best mean values of the maximum deviations were obtained with the algorithms {\tt Cartesian into Geodetic I} and {\tt Cartesian into Geodetic II}.  Nevertheless, the CPU times obtained in {\tt Maple} were high (as other formula processing systems, {\tt Maple} runs in the interpreter mode, and therefore, it runs slow). 

For this reason, the definitive implementation of the aforementioned algorithms has been performed in a compiler-type programing language, specifically in C++. The definitive CPU running times, in C++, differ in an order of magnitude 3 from the initial ones, in {\tt Maple}. The results have been obtained working with double precision, on an Intel(R) Core(TM) i7-7700K CPU @ 4.20 GHz x 8 processor with 62,8GB of RAM.

 
 The considered celestial bodies, together with their shape parameters ($a_x$, $a_y$ and $a_z$ respectively) (see \cite{ligas}, \cite{tierra}, \cite{Seidelmann2007}, \cite{seidelmann02}, \cite{wu}) are as follows: 

\begin{table}[H]
\begin{center}
\begin{tabular}{@{} |l|c|c|c|c| @{}}
\toprule
\small  Celestial body   &  \small  $a_x$  &  \small  $a_y$  &  \small  $a_z$   \\  
\midrule
\small Ariel &  \small 581.1  &  \small 577.9  &  \small 577.7  \\ 
\small Earth &  \small 6378.173435  &  \small 6378.1039  &  \small 6356.7544  \\ 
\small Enceladus  &   \small 256.6  &  \small 251.4  &  \small 248.3  \\ 
\small Europa  &  \small 1564.13  &  \small 1561.23  &  \small 1560.93  \\  
\small Io  &  \small 1829.4  &  \small 1819.3  &  \small 1815.7  \\  
\small Mars &  \small 3394.6  &  \small 3393.3  &  \small 3376.3  \\ 
\small Mimas  &  \small 207.4  &  \small 196.8  &  \small 190.6  \\ 
\small Miranda &  \small 240.4  &  \small 234.2  &  \small 232.9  \\ 
\small Moon  &  \small 1735.55  &  \small 1735.324  &  \small 1734.898  \\  
\small Tethys &  \small 535.6  &  \small 528.2  &  \small 525.8  \\ 
\bottomrule
\end{tabular}
\end{center}
\caption{\small Semiaxes (in km) of the considered celestial bodies}\label{ejes}
\end{table}

Following \cite{ligas}, we consider the points in the first octant defined by the geodetic coordinates $(\varphi_i,\lambda_j,h_k)$, where $\varphi_i = \dfrac{i \pi}{720}$ radians, $i=1 \ldots 359$, $\lambda_j = \dfrac{j \pi}{720}$ radians, $j=1 \ldots 359$, $h_k  = ka_z $ km, $k\in\{0, \pm\dfrac{1}{50}, \pm\dfrac{1}{25}, \pm\dfrac{1}{15}, \pm\dfrac{1}{10}\}  
\label{tests}$. For each point, we compute its Cartesian coordinates from (\ref{PG}) and apply the corresponding algorithm for computing its geodetic coordinates, comparing the obtained values with the initial ones. We have excluded from the points considered for the numerical tests the following cases: $\varphi_0=0$, in which case $Z_G=0$ and $X_GY_G>0$ and Case 3 of Ligas' method can't be applied, as the Jacobian is singular; $\varphi_{360}=\frac{\pi}{2}$, in which case $X_G=Y_G=0$ and the longitude is undefined (see \cite{muller}); $\lambda_0=0$, in which case $Y_G=0$ and $X_G>0$ and Case 2 of Ligas' method can't be applied; and $\lambda_{360}=\frac{\pi}{2}$, in which case $X_G=0$ and $Y_G>0$ and Case 1 of Ligas' method can't be applied. Therefore, we considered, for each algorithm and each celestial body, 359 latitudes, 359 longitudes and 9 heights along the normal, i.e. a total of 1159929 different points. The averaged CPU times are computed by applying the corresponding algorithm once to all these points, and including the computation of the maximum errors.

The following tables present the maximum differences in absolute value between the real, known geodetic coordinates and the computed ones, on a base-10 log scale, together with the mean CPU running times in seconds. A logaritmic scale is a nonlinear scale often used when analyzing a very wide or narrow range of positive quantities. In the following tables, in the second, third and fourth columns, instead of displaying the maximum errors as $\varepsilon=10^{a}$, where $a$ is some negative real number, we display $log_{10}(\varepsilon)=a$.

\begin{table}[H]
\begin{center}
\begin{tabular}{@{} |l|c|c|c|c| @{}}
\toprule
{\small Celestial body }   &  {\small Max. err. $\lambda$}  &  {\small Max. err. $\varphi$}  &  {\small Max. err. $h$}  &  {\small Time    }  \\  
\midrule
{\small Ariel }    &      
{\small \bf{ -18.789}} / {\small -18.664}      &      
{\small -18.664} / {\small  -18.664}       &      
{\small\bf{ -15.634}} / {\small { -15.400}}        &      
{\small 1.179343} / {\small \bf{1.093714}}        \\ 
{\small Earth }     &     
{\small -18.664} / {\small -18.664}       &      
{\small -18.664} / {\small -18.664}        &      
{\small \bf{-14.700}} /  {\small -14.500}        &      
{\small 1.191914} / {\small \bf{1.093113}}        \\ 
{\small Enceladus }     &      
{\small \bf{-18.420}} / {\small -15.940}        &      
{\small {\bf -18.311}} / {\small { -17.885}}      &     
{\small  {\bf -14.512}} / {\small -11.873}         &      
{\small 1.296452} / {\small \bf{1.190336}}          \\ 
{\small Europa }     &       
{\small -18.664} / {\small -18.664}       &       
{\small -18.567} / {\small {\bf -18.664}}        &      
{\small \bf{-15.244}} /  {\small -12.768}        &       
{\small 1.182543} / {\small \bf{1.092273}}        \\ 
{\small Io }     &     
{\small \bf{-18.789}} / {\small -18.664}       &      
{\small  \bf{-18.664}} / \bf{\small -18.664}        &      
{\small \bf{-15.277}} / {\small -14.767}        &      
{\small 1.183251} / {\small \bf{1.092522}}        \\ 
{\small Mars }    &      
{\small -18.664} / {\small -18.567}       &      
{\small \bf{-18.664}} / {\small\bf{ -18.664} }       &      
{\small \bf{-15.000}} / {\small -14.816}      &      
{\small 1.228557} / {\small \bf{1.103634}}        \\ 
{\small Mimas }     &      
{\small -17.698} / {\small \bf{ -18.664}}       &    
{\small -17.550} / {\small \bf{-18.664}}        &      
{\small -14.142} / {\small \bf{ -15.559}}        &      
{\small 1.166641} / {\small \bf{1.152859}}        \\ 
{\small Miranda }     &      
{\small {\bf -18.266}} / {\small -15.793}       &      
{\small {\bf -18.186}} /  {\small {-17.862}}      &      
{\small {\bf -14.426}} / {\small -11.873}        &      
{\small 1.166864} / {\small \bf{1.107854}}        \\ 
{\small Moon }    &      
{\small {\bf -18.789}} / {\small -18.664}       &      
{\small \bf{-18.664}} / {\small{\bf -18.664}}        &     
{\small \bf{-15.244}} / {\small -15.045}       &      
{\small 1.180139} / {\small \bf{1.093541}}        \\ 
{\small Tethys}      &      
{\small \bf{-18.664}} / {\small -17.311}       &     
{\small -18.664} / {\small \bf{-18.664}}        &      
{\small  {\bf -15.371}} / {\small -12.331}         &      
{\small \bf{1.175097}} / {\small 1.196335}       \\ 
\midrule
{\small {\it Mean values} }    &      
\small{ \textbf{\textit{-18.540}}} /  {{\it -17.959}}      &     
\small{\textit{ -18.460}} / \textbf{\textit{ -18.506}}      &      
\small{\textbf{\textit{ -14.955}}} /  {\it -13.893}      &       
\small{\it 1.1950801} / {\textbf{\textit{ 1.1216181}}}        \\  
\bottomrule
\end{tabular}
\end{center}
\caption{\small Results obtained by applying the algorithms {\tt Cartesian into Geodetic I} / {\tt Cartesian into Geodetic II} implemented in C++}\label{BellGB-3decimales}
\end{table}

\begin{table}[H]
\begin{center}
\begin{tabular}{@{} |l|c|c|c|c| @{}}
\toprule
\tiny Celestial body    & \tiny Max. err. $\lambda$  & \tiny Max. err. $\varphi$  & \tiny Max. err. $h$  & \tiny Time      \\  
\midrule
\tiny Ariel     &      
\tiny -17.775 / -17.664 / -18.488       &     
\tiny -18.337 / \bf{-18.789} / -18.664        &     
\tiny -13.664 / -13.662 / -13.663         &     
\tiny 1.270501 / 1.281323 / 1.270980         \\ 
\tiny Earth     &      
\tiny \bf{-18.789} / \bf{-18.789} / \bf{-18.789}        &     
\tiny -18.664 / -18.664 / \bf{-18.789}         &     
\tiny -14.552 / -14.627 / -14.612         &     
\tiny 1.275257 / 1.274793 / 1.270144         \\ 
\tiny Enceladus     &     
\tiny -14.804 / -15.169 / -17.580         &     
\tiny -17.145 / -17.139 / -17.146         &     
\tiny -13.305 / -13.304 / -13.299          &     
\tiny 1.279295 / 1.271683 / 1.270610          \\ 
\tiny Europa     &     
\tiny \bf{-18.789} / \bf{-18.789} / \bf{-18.789}        &     
\tiny \bf{-18.664} / \bf{-18.664} / \bf{-18.664}        &     
\tiny -14.084 / -14.084 / -14.083       &     
\tiny 1.268811 / 1.298762 / 1.322419         \\ 
\tiny Io     &     
\tiny -17.446 / -17.488 / -18.789       &     
\tiny \bf{-18.664} / \bf{-18.664} / \bf{-18.664}        &     
\tiny -14.148 / -14.148 / -14.151         &     
\tiny 1.271437 / 1.271444 / 1.271251         \\  
\tiny Mars     &     
\tiny \bf{-18.789} / \bf{-18.789} / \bf{-18.789}       &     
\tiny \bf{-18.664} / \bf{-18.664} / \bf{-18.664}         &     
\tiny -14.372 / -14.366 / -14.372        &     
\tiny 1.269966 / 1.277228 / 1.273551       \\ 
\tiny Mimas     &     
\tiny -16.583 / -14.260 / -16.780        &     
\tiny -16.851 / -16.786 / -16.851         &     
\tiny -13.185 / -13.186 / -13.183         &     
\tiny 1.348541 / 1.340201 / 1.340244        \\ 
\tiny Miranda     &     
\tiny -14.625 / -15.225 / -17.534        &     
\tiny -17.257 / -17.257 / -17.257         &     
\tiny -13.274 / -13.272 / -13.270         &    
 \tiny 1.270835 / 1.273983 / 1.272070        \\ 
\tiny Moon     &     
\tiny \bf{-18.789} / \bf{-18.789} /\tiny \bf{ -18.789}       &     
\tiny \bf{-18.664} / \bf{-18.664} / \bf{-18.664}        &     
\tiny -14.123 / -14.122 /\tiny -14.122         &     
\tiny 1.238008 / 1.240675 /\tiny 1.245733        \\ 
\tiny Tethys      &     
\tiny -15.733 / -16.062 / -17.886       &     
\tiny -17.972 / -18.187 / \bf{-18.664}         &     
\tiny -13.627 / -13.625 /-13.624          &     
\tiny 1.273102 / 1.277854 / 1.273062        \\ 
\midrule
{\tiny \it Mean values}    &     
\tiny{\it -17.212} / {\it -17.102} / {\it -18.221}       &    
\tiny{\it -18.088} / {\it -18.148} / {\it -18.203}       &     
\tiny{\it -13.833} / {\it -13.840} / {\it -13.838}       &      
\tiny{\it 1.276575} / {\it 1.280795} / {\it 1.281006}          \\
\bottomrule
\end{tabular}
\end{center}
\caption{Results obtained by applying Case 1 / Case 2 / Case 3 of Ligas' method implemented in C++}\label{Ligas123-3decimales}
\end{table}

\begin{table}[H]
\begin{center}
\begin{tabular}{@{} |l|c|c|c|c| @{}}
\toprule
\tiny Celestial body    & \tiny Max. err. $\lambda$  & \tiny Max. err. $\varphi$  & \tiny Max. err. $h$  & \tiny Time      \\  
\midrule
\tiny Ariel     &     
\tiny -9.148  /  \tiny -9.151 / \tiny -9.193    &     
\tiny -11.515 /\tiny -11.533 /\tiny -11.470      &     
\tiny -8.757 /\tiny -8.784 /\tiny -8.760       &      
\tiny 1.397315 /\tiny 1.391594 / \tiny 1.359541     \\ 
\tiny Earth     &    
\tiny  -12.760 / \tiny -12.902 /\tiny -12.898      &     
\tiny -12.768 /\tiny -12.788 /\tiny -12.782        &     
\tiny -9.011 /\tiny -9.034 /\tiny -9.029        &     
\tiny 1.381216 /\tiny 1.363576 /\tiny 1.366670        \\ 
\tiny Enceladus     &     
\tiny -8.062 /\tiny -8.243 /\tiny -8.216        &     
\tiny -10.414 /\tiny -10.609 /\tiny -10.333        &     
\tiny -8.020 /\tiny -8.221 /\tiny -7.992         &     
\tiny 1.623761 /\tiny 1.531733 /\tiny 1.587326          \\ 
\tiny Europa     &     
\tiny -10.044 /\tiny -10.112 /\tiny -10.062        &     
\tiny -12.478 /\tiny -12.479 /\tiny -12.378         &     
\tiny -9.214 /\tiny -9.256 /\tiny -9.242         &     
\tiny 1.161976 /\tiny 1.158426 /\tiny 1.170673          \\ 
\tiny Io     &     
\tiny -9.555 /\tiny -9.697 /\tiny -9.637         &     
\tiny -11.882 /\tiny -12.027 /\tiny -11.868         &     
\tiny -8.650 /\tiny -8.782 /\tiny -8.723         &     
\tiny 1.429542 /\tiny 1.403343 /\tiny 1.415738        \\ 
\tiny Mars     &     
\tiny -10.673 /\tiny -11.691 /\tiny -11.723        &     
\tiny -12.302 /\tiny -12.346 /\tiny -12.325         &     
\tiny -8.793 /\tiny -8.820 /\tiny -8.976         &     
\tiny 1.413745 /\tiny 1.376237 /\tiny 1.392895         \\ 
\tiny Mimas     &     
\tiny -7.592 /\tiny -7.812 /\tiny -7.709        &     
\tiny -9.965 /\tiny -10.162 /\tiny -9.758         &     
\tiny -7.633 /\tiny -7.826 /\tiny -7.560         &     
\tiny 1.951112 /\tiny 1.740943 /\tiny 1.858730         \\ 
\tiny Miranda     &     
\tiny -8.031 /\tiny -8.128 /\tiny -8.070        &     
\tiny -10.394 /\tiny -10.498 /\tiny -10.325         &     
\tiny -8.036 /\tiny -8.117 /\tiny -8.008         &     
\tiny 1.600380 /\tiny 1.574738 /\tiny 1.638333         \\ 
\tiny Moon     &     
\tiny -12.961 /\tiny -11.250 /\tiny -11.259        &     
\tiny -13.227 /\tiny -13.864 /\tiny -13.228         &     
\tiny -10.212 /\tiny -10.412 /\tiny -10.213         &     
\tiny 1.222829 /\tiny 1.220730 /\tiny 1.229268         \\ 
\tiny Tethys      &     
\tiny -8.616 /\tiny -8.721 /\tiny -8.738        &     
\tiny -10.961 /\tiny -11.110 /\tiny -10.921         &     
\tiny -8.265 /\tiny -8.386 /\tiny -8.675          &     
\tiny 1.533092 /\tiny 1.498667 /\tiny 1.534345        \\ 
\midrule
{\tiny \it Mean values}    &      
\tiny {\it -9.744} / {\it -9.771} / {\it -9.751}       &     
\tiny {\it -11.591} /  {\it -11.742} / {\it -11.539}       &      
\tiny {\it -8.659} / {\it -8.764} / {\it -7.113}      &       
\tiny {\it 1.471497} / {\it 1.425999} / {\it 1.455352}          \\  
\bottomrule
\end{tabular}
\end{center}
\caption{Results obtained by applying Case 1 / Case 2 / Case 3 of Feltens' method implemented in C++}\label{Feltens123-3decimales}
\end{table}

In all the considered case studies,  the best and second best running times are obtained with the algorithms {\tt Cartesian into Geodetic II} and {\tt Cartesian into Geodetic I}. Moreover, the best and second best mean values of the maximum deviations obtained in the 10 case studies correspond to our algorithms, except for the second best mean value of the maximum deviation of the longitude (which corresponds to the Case 3 of Feltens' method). The three best results are presented in the following table:

\begin{table}[H]
\begin{center}
\begin{tabular}{@{} |l|c|c|c|c| @{}}
\toprule
\tiny Position  & \tiny  Max. err. $\lambda$  & \tiny  Max. err. $\varphi$  & \tiny  Max. err. $h$  & \tiny  Time      \\  
\midrule
\tiny  Best result    &     
\tiny  {\tt Cartesian into Geodetic I}    &     
\tiny {\tt Cartesian into Geodetic II}      &     
\tiny  {\tt Cartesian into Geodetic I}       &      
\tiny   {\tt Cartesian into Geodetic II}     \\ 
\tiny Second best result     &    
\tiny  Case 3 of Feltens' method      &     
\tiny {\tt Cartesian into Geodetic I}        &     
\tiny {\tt Cartesian into Geodetic II}        &     
\tiny {\tt Cartesian into Geodetic I}        \\ 
\tiny Third best result     &     
\tiny {\tt Cartesian into Geodetic II}        &     
\tiny Case 3 of Feltens' method        &     
\tiny Case 2 of Feltens' method         &     
\tiny Case 2 of Feltens' method          \\ 

\midrule
\bottomrule
\end{tabular}
\end{center}
\caption{Ranking of the three best results in computing the mean values of the maximum deviations and CPU running times}\label{ranking}
\end{table}

These results show that our approaches improve the methods presented in \cite{feltens09} and \cite{ligas}, in terms of both efficiency and accuracy.

\section{Conclusions and further work}\label{SecConclusions}

We have presented two efficient algorithms for the transformation of Cartesian coordinates into geodetic coordinates, for a triaxial reference ellipsoid. Each algorithm is based on the numeric computation of the unique real positive root of a degree 6 polynomial, symbolically generated.  

One of the main topics of our further work consists in studying the case of the hyperboloidal coordinates considered for triaxial reference hyperboloids and providing a similar approach for the transformation of the cartesian coordinates. From the geometric and algebraic points of view, both problems are closely related.
This problem hasn't been tackled before and furthermore there are very few approaches for the biaxial case (see \cite{diaznecula14} for a closed form solution and \cite{feltens11} for a iterative solution). 
 
\section{Acknowledgments}\label{SecAgrad}

The first and third authors are partially supported by FEDER/Ministerio de Ciencia, Innovaci\'on y Universidades - Agencia Estatal de Investigaci\'on/MTM2017-88796-P (Symbolic Computation: new challenges in Algebra and Geometry together with its applications). The second author is partially funded by the project TIN2017-86885-R cofinanced by the EU Feder program. The third author wish to kindly thank Prof. Enrique D. Fern\'andez-Nieto for his support.


%

\bibliographystyle{plain}
\bibliography{tri-gema-ioana-VF-CG}

\begin{thebibliography}{10}

\bibitem{awange18}
J.L. Awange and B.~Pal\'ancz.
\newblock {\em {G}eospatial {A}lgebraic {C}omputations: {T}heory and
  {A}pplications}.
\newblock Springer, 2018.

\bibitem{bell20}
R.J.T. Bell.
\newblock {\em An elementary treatise on coordinate geometry of three
  dimensions}.
\newblock Macmillan, London, 1920.

\bibitem{bowring76}
B.~R. Bowring.
\newblock Transformation from spatial to geographical coordinates.
\newblock {\em Survey Review}, 23(181):323--327, 1976.

\bibitem{bursa71}
M.~Bur\v{s}a and E.~Buchar.
\newblock On the triaxiality of the {E}arth on the basis of {S}atellite {D}ata.
\newblock {\em Studia Geophysica et Geodaetica}, 15(3--4):228--240, 1971.

\bibitem{bursa80}
M.~Bur\v{s}a, Z.~\v{S}{\'\i}ma, and J.~P{\'\i}cha.
\newblock Tri-axiality of the {E}arth, the {M}oon and {M}ars.
\newblock {\em Studia Geophysica et Geodaetica}, 24(3):211--217, 1980.

\bibitem{civicioglu12}
P.~Civicioglu.
\newblock Transforming geocentric cartesian coordinates to geodetic coordinates
  by using differential search algorithm.
\newblock {\em Computers $\&$ Geosciences}, 46(1):229--247, 2012.

\bibitem{cox07}
D.~A. Cox, J.~Little, and D.~O'Shea.
\newblock {\em Ideals, Varieties, and Algorithms: An Introduction to
  Computational Algebraic Geometry and Commutative Algebra, Third edition
  (Undergraduate Texts in Mathematics)}.
\newblock Springer Verlag, 2007.

\bibitem{diaznecula14}
G.M. D{\'\i}az-Toca and I.~Necula.
\newblock Direct symbolic transformation from 3{D} cartesian into hyperboloidal
  coordinates.
\newblock {\em Applied Mathematics and Computation}, 228:349--365, 2014.

\bibitem{Eberly2006DistanceFA}
David Eberly.
\newblock Distance from a {P}oint to an {E}llipse, an {E}llipsoid, or a
  {H}yperellipsoid.
\newblock 2006.

\bibitem{feltens08}
J.~Feltens.
\newblock Vector method to compute azimuth, elevation, ellipsoidal normal, and
  the {C}artesian $({X},{Y},{Z})$ to geodetic $(\phi,\lambda,h)$
  transformation.
\newblock {\em Journal of Geodesy}, 82(8):493--504, 2008.

\bibitem{feltens09}
J.~Feltens.
\newblock Vector method to compute the {C}artesian $({X},{Y},{Z})$ to geodetic
  $(\phi,\lambda,h)$ transformation on a triaxial ellipsoid.
\newblock {\em Journal of Geodesy}, 83(2):129--137, 2009.

\bibitem{feltens11}
J.~Feltens.
\newblock Hyperboloidal coordinates: transformations and applications in
  special constructions.
\newblock {\em Journal of Geodesy}, 85(4):239--254, 2011.

\bibitem{fukushima99}
T.~Fukushima.
\newblock Fast transform from geocentric to geodetic coordinates.
\newblock {\em Journal of Geodesy}, 73(11):603--610, 1999.

\bibitem{fukushima06}
T.~Fukushima.
\newblock Transformation from {C}artesian to geodetic coordinates accelerated
  by {H}alley's method.
\newblock {\em Journal of Geodesy}, 79(12):689--693, 2006.

\bibitem{laloirene09}
L.~Gonzalez-Vega and I.~Polo-Blanco.
\newblock A symbolic analysis of {V}ermeille and {B}orkowski polynomials for
  transforming 3{D} {C}artesian to geodetic coordinates.
\newblock {\em Journal of Geodesy}, 83(11):1071--1081, 2009.

\bibitem{grafarend14}
E.~W. Grafarend, R.-J. You, and R.~Syffus.
\newblock {\em Map Projections. Cartographic Information Systems}.
\newblock Springer-Verlag Berlin Heidelberg, 2014.

\bibitem{hart94}
J.C. Hart.
\newblock {\em Distance to an ellipsoid}, pages 113--119.
\newblock Morgan Kaufamann, Menlo Park, 1994.

\bibitem{heiskanen62}
W.~A. Heiskanen.
\newblock Is the {E}arth a {T}riaxial {E}llipsoid?
\newblock {\em Jornal of Geophisical Research}, 67(1):321--327, 1962.

\bibitem{ligas}
M.~Ligas.
\newblock Cartesian to geodetic coordinates conversion on a triaxial ellipsoid.
\newblock {\em Journal of Geodesy}, 86(4):249--256, 2012.

\bibitem{mignotte}
M.~Mignotte.
\newblock {\em Mathematics for {C}omputer {A}lgebra}.
\newblock Springer-Verlag New York, 1992.

\bibitem{muller}
B.~M\"uller.
\newblock {\em Kartenprojektionen des dreiachsigen {E}llipsoids}.
\newblock PhD thesis, University of Stuttgart, Germany, 1991.

\bibitem{tierra}
G.~Schliephake.
\newblock Berechnungen auf dem dreiachsigen {E}rdellipsoid nach {K}rassowski.
\newblock {\em Ver- messungstechnik}, 4(7-10), 1956.

\bibitem{Seidelmann2007}
P.~Kenneth Seidelmann, B.~A. Archinal, M.~F. A'hearn, A.~Conrad, G.~J.
  Consolmagno, D.~Hestroffer, J.~L. Hilton, G.~A. Krasinsky, G.~Neumann,
  J.~Oberst, P.~Stooke, E.~F. Tedesco, D.~J. Tholen, P.~C. Thomas, and I.~P.
  Williams.
\newblock Report of the iau/iag working group on cartographic coordinates and
  rotational elements: 2006.
\newblock {\em Celestial Mechanics and Dynamical Astronomy}, 98(3):155--180,
  Jul 2007.

\bibitem{seidelmann02}
P.K. Seidelmann, V.K. Abalakin, M.~Bursa, M.E. Davies, C.~De Bergh, J.H.
  Lieske, J.~Oberst, J.L. Simon, E.M. Standish, P.~Stooke, and P.C. Thomas.
\newblock Report of the {IAU}/{IAG} {W}orking {G}roup on {C}artographic
  {C}oordinates and {R}otational {E}lements of the {P}lanets and {S}atellites:
  2000.
\newblock {\em Celestial Mechanics and Dynamical Astronomy}, 82:83--110, 2002.

\bibitem{shu10}
C.~Shu and F.~Li.
\newblock An iterative algorithm to compute geodetic coordinates.
\newblock {\em Computers $\&$ Geosciences}, 36(9):1145--1149, 2010.

\bibitem{soler12}
T.~Soler, J.~Y. Han, and N.~D. Weston.
\newblock Alternative transformation from {C}artesian to geodetic coordinates
  by least squares for {G}{P}{S} georeferencing applications.
\newblock {\em Computers $\&$ Geosciences}, 42(1):100--109, 2012.

\bibitem{souchay03}
J.~Souchay, M.~Folgueira, and S.~Bouquillon.
\newblock Effects of the triaxiality on the rotation of celestial bodies:
  {A}pplication to the {E}arth, {M}ars and {E}ros.
\newblock {\em Earth, Moon and Planets}, 93(2):107--144, 2003.

\bibitem{turner11}
J.D. Turner.
\newblock Universal {A}lgorithm for {I}nverting the {C}artesian to {G}eodetic
  {T}ransformation.
\newblock {\em Journal of the Astronautical Sciences}, 58(3):429--443, 2011.

\bibitem{vermeille02}
H.~Vermeille.
\newblock Direct transformation from geocentric coordinates to geodetic
  coordinates.
\newblock {\em Journal of Geodesy}, 76(8):451--454, 2002.

\bibitem{wu}
S.~S.~C. {Wu}.
\newblock A method of defining topographic datums of planetary bodies.
\newblock {\em Annales de Geophysique}, 37:147--160, March 1981.

\end{thebibliography}

\end{document}